\documentclass{amsart}
\usepackage{tikz}
\usetikzlibrary{fadings,arrows}

\usepackage{amsmath, amssymb, mathrsfs, verbatim, multirow}
\usepackage[all]{xy}
\usepackage{pifont}
\usepackage{float}
\DeclareMathOperator{\inv}{in}
\newtheorem{Teo}{Theorem}[section]
\newtheorem{Prop}[Teo]{Proposition}
\newtheorem{Lema}[Teo]{Lemma}
\newtheorem{Cor}[Teo]{Corollary}

\theoremstyle{definition}
\newtheorem{Def}[Teo]{Definition}

\newtheorem{Obs}[Teo]{Remark}

\newcommand{\N}{\mathbb{N}}

\newcommand{\lra}{\longrightarrow}
\newcommand{\Lra}{\Longrightarrow}
\newcommand{\VR}{\mathcal{O}}

\newcommand{\MI}{\mathfrak{m}}

\begin{document}
\title[Generators]{Generators for extensions of valuation rings}

\author{Josnei Novacoski}
\address{Departamento de Matem\'{a}tica,         Universidade Federal de S\~ao Carlos, Rod. Washington Luís, 235, 13565--905, S\~ao Carlos -SP, Brazil}
\email{josnei@ufscar.br}

\thanks{During the realization of this project the authors were supported by a grant from Funda\c{c}\~ao de Amparo \`a Pesquisa do Estado de S\~ao Paulo (process number 2017/17835-9).}

\begin{abstract}
For a finite valued field extension $(L/K,v)$ we describe the problem of finding sets of generators for the corresponding extension $\VR_L/\VR_K$ of valuation rings. The main tool to obtain such sets are complete sets of (key) polynomials. We show that when the initial index coincides with the ramification index, sequences of key polynomials naturally give rise to sets of generators. We use this to prove Knaf's conjecture for pure extensions. 
\end{abstract}

\maketitle
\section{Introduction}
Let $(L/K,v)$ be a finite extension of valued fields and denote by $\VR_L$ and $\VR_K$ the corresponding valuation rings. The main purpose of this paper is to describe possible sets of generators for $\VR_L$ as an $\VR_K$-algebra or an $\VR_K$-module.

The motivation for this comes from various areas. The first motivation comes from the local uniformization problem in positive characteristic (see \cite{Nov4} for more details). This problem can be seen as the resolution of a fixed singularity on an algebraic variety along a fixed valuation. Most of the programs to solve this problem rely implicitly on having good sets of generators for extensions of valuation rings.  

Another motivation comes from the study of the module of K\"ahler differentials of the extension $\VR_L/\VR_K$. In order to compute such module, it is necessary to present a set of generators and relations for $\VR_L$ over $\VR_K$. In \cite{CKR} and \cite{CK}, for each extension $(L/K,v)$ of prime degree, the authors present a set of generators for $\VR_L$ over $\VR_K$ and use them to compute the corresponding module of K\"ahler differentials. In \cite{NS2023} a similar idea is used. More precisely, for any simple algebraic extension $(L/K,v)$ of valued fields, if the \emph{ramification index} is one, then for every \emph{sequence of key polynomials} one can naturally construct a set of generators for the extension $\VR_L$ over $\VR_K$. 

The first result from this paper is a generalization of \cite[Proposition 3.5]{NS2023}. Let $\Gamma$ be an ordered abelian group and $\Delta$ a subgroup of $\Gamma$. For a subset $S$ of $\Gamma$ we denote by $|S|$ the cardinality of $S$. Also, if $\gamma\in\Gamma$, then we use $\gamma<S$ to express that $\gamma<s$ for every $s\in S$. The \textbf{initial index} of $\Delta$ in $\Gamma$ is defined as
\[
\epsilon(\Gamma|\Delta)=|\{\gamma\in \Gamma\mid 0\leq \gamma< \Delta_{>0}\}|,
\]
where $\Delta_{>0}$ denotes the set of all positive elements in $\Delta$. For a finite extension of valued fields $(L/K,v)$ we define $\epsilon(L/K,v):=\epsilon(vL|vK)$. It is easy to show that
\begin{equation}\label{equalieepsilon}
\epsilon(L/K,v)\leq e(L/K,v):=(vL:vK),
\end{equation}
and the equality holds if $e(L/K,v)=1$. The equality in \eqref{equalieepsilon} allows us to prove similar results as when $e(L/K,v)=1$.

Let $q\in K[x]\setminus K$ and $f\in K[x]$. There exist uniquely determined $f_0,\ldots,f_r\in K[x]$ with
$\deg(f_\ell)<\deg(q)$ for every $\ell$, $0\leq \ell\leq r$, such that
\[
f=f_0+f_1q+\ldots+f_rq^r.
\]
This expression is called the \textbf{$q$-expansion of $f$}. Let $\mu$ be a valuation of $K[x]$. For a monic polynomial $q\in K[x]\setminus K$ the \textbf{truncation of $\mu$ at $q$} is defined as
\[
\mu_q(f)=\min_{0\leq \ell\leq r}\left\{\mu\left(f_\ell q^\ell\right)\right\},
\]
where $f=f_0+f_1q+\ldots+f_rq^r$ is the $q$-expansion of $f$.

A set $\textbf{Q}\subseteq K[x]$ is called a {\bf complete set for $\mu$} if for every $f\in K[x]$ there exists $q\in \textbf{Q}$ such that
\begin{equation}\label{eqabotdegpol}
\deg(q)\leq\deg(f)\mbox{ and }\mu(f)=\mu_q(f).
\end{equation}

For a simple algebraic extension of valued fields $(L/K,v)$ and a generator $\eta$ of $L/K$, the valuation $\nu_\eta$ of $K[x]$ \textbf{defined by $v$ and $\eta$} is
\[
\nu_\eta f:=v(f(\eta)).
\]
A complete set for $(L/K,v)$ is defined as a set of the form $\{Q_i(\eta)\}_{i\in I}$ such that $\{Q_i\}_{i\in I}$ is a complete set for $\nu_\eta$.

For a set $I$ we denote by $\N_0^I$ the set of mappings $\lambda:I\lra \N_0$ such that $\lambda(i)=0$ for all but finitely many $i\in I$ (here $\N_0$ denotes the set of non-negative integers). Let $\textbf{Q}=\{Q_i\}_{i\in I}\subseteq K[x]$ be a set of polynomials indexed by $I$. For $\lambda\in \N_0^I$ and $\eta\in L$ we denote
\[
\textbf{Q}^\lambda:=\prod_{i\in I}Q_i^{\lambda(i)}\in K[x]\mbox{ and }\textbf{Q}(\eta)^\lambda:=\prod_{i\in I}Q_i(\eta)^{\lambda(i)}\in L.
\]
The following result is a generalization of \cite[Proposition 3.5]{NS2023}.
\begin{Teo}\label{mainthem}
Let $(L/K,v)$ be a simple algebraic extension of valued fields. Assume that $\epsilon(L/K,v)= e(L/K,v)$ and take any complete set $\{Q_i(\eta)\}_{i\in I}$ for $(L/K,v)$. For each $\lambda\in\N_0^I$ there exists $a_\lambda\in K$ such that $\VR_L$ is generated by
\[
\left\{\left.\frac{\textbf Q(\eta)^\lambda}{a_\lambda}\ \right |\ \lambda\in\N_0^I\right\}
\]
as an $\VR_K$-module. 
\end{Teo}
When $e(L/K,v)=1$ the situation is much simpler. In that case, \cite[Proposition 3.5]{NS2023} shows that for each $i\in I$ there exists $a_i$ such that
\[
\VR_L=\VR_K\left[\left.\frac{Q_i(\eta)}{a_i}\ \right |\ i\in I\right].
\]
In particular, for each $\lambda\in\N_0^I$ the corresponding element $a_\lambda$, as in Theorem \ref{mainthem}, can be chosen to be $\textbf{a}^\lambda$ (where $\textbf{a}:=\{a_i\}_{i\in I}$).

We also present an application of Theorem \ref{mainthem}. Namely, we use it to prove Knaf's conjecture for pure extensions. Let $d(L/K,v)$ be the \emph{defect} of $(L/K,v)$. For any subring $A$ of $\VR_L$ we will denote by $A_v$ the localization of $A$ at $A\cap \MI_L$ (here $\MI_L$ denotes the maximal ideal of $\VR_L$). We say that $\VR_L$ is \textbf{essentially finitely generated over $\VR_K$} if there exist $b_1,\ldots, b_r\in \VR_L$ such that
\[
\VR_L=\VR_K[b_1,\ldots,b_r]_v.
\]
Knaf proved that if $\VR_L$ is essentially finitely generated over $\VR_K$, then
\begin{equation}\label{knafconsudtion}
e(L/K,v)=\epsilon(L/K,v)\mbox{ and }d(L/K,v)=1.
\end{equation}
He conjectured that the converse is also true, i.e., that if \eqref{knafconsudtion} is satisfied, then $\VR_L$ is essentially finitely generated over $\VR_K$.

This conjecture was explored in \cite{Nov12}. There, it was proved that if $K$ is the quotient field of an excellent two-dimensional local domain and \eqref{knafconsudtion} is satisfied, then $\VR_L$ is essentially finitely generated over $\VR_K$. Also in \cite{Nov12} it was proved that Knaf's conjecture is true if $K$ is an algebraic function field over a field $k$, $v$ is an \textit{Abhyankar valuation} of $K$, and the residue field of $(L,v)$ is a separable extension of $k$. In \cite{CutKnaf}, Cutkosky generalized this latter result, without the assumption on the residue field of $(L,v)$. In \cite{Datta2021}, Datta presented a proof of Knaf's conjecture in full generality.

In this paper we present an alternative proof of Knaf's conjecture for \emph{pure extensions}. Let $(L/K,v)$ be a simple algebraic extension of valued fields of degree $n$. Fix a generator $\eta$ of $L/K$ and consider the valuation $\nu_\eta$ on $K[x]$ defined by $v$ and $\eta$. For each $m$, $1\leq m\leq n$, we define
\[
\Psi_m:=\{Q\in K[x]\mid \deg(Q)=m\mbox{ and }Q\mbox{ is a key polynomial for }\nu_\eta\}
\]
(see Section \ref{keypolydefectformula} for the definition of key polynomials). 
\begin{Def}
We say that $(L/K,v)$ is pure in $\eta$ if $\Psi_m=\emptyset$ for every $m$, $1<m<n$. We will say that $(L/K,v)$ is pure if it is pure in $\eta$ for some generator $\eta$ of $L/K$.
\end{Def}

\begin{Prop}\label{Propimportan}
Assume that $(L/K,v)$ is pure and $e:=e(L/K,v)=\epsilon(L/K,v)$. If $d(L/K,v)=1$, then there exist a generator $\eta$ of $L/K$ and $a_1,\ldots,a_e\in K$ such that
\[
\VR_L=\VR_K\left[\frac{\eta}{a_1},\ldots,\frac{\eta^e}{a_e}\right]_v.
\]
\end{Prop}
The proof of Proposition \ref{Propimportan} uses the theory of key polynomials and the defect formula (see Section \ref{keypolydefectformula}).

\par\medskip
\textbf{Acknowledgements.} I would like to thank the anonymous referee for a careful reading, for providing useful suggestions and for pointing out a few mistakes in an earlier version of this paper.

\section{Notation}
In this paper we will use the letter $v$ to denote a valuation of $K$ or its extensions to algebraic extensions of $K$. The notation $\nu_\eta$ will be used for the valuation of $K[x]$ defined by $v$ and $\eta$ where $(L/K,v)$ is a simple algebraic extension and $\eta$ a generator of $L/K$. For a general valuation of $K[x]$, we will use the letter $\mu$.

For a valued field $(K,v)$ we will denote by $\VR_K$ the valuation ring, by $vK$ the value group and by $Kv$ the residue field of $v$. Also, for $b\in K$ we denote by $vb$ or $v(b)$ the value of $b$ in $vK$. If $b\in \VR_K$, then we denote by $bv$ the residue of $b$ in $Kv$.

For a finite valued field extension $(L/K,v)$ we will denote by $e(L/K,v)$ and $f(L/K,v)$ the ramification and inertial indices of
$(L/K,v)$, respectively:
\[
e(L/K,v)=(vL:vK)\mbox{ and }f(L/K,v)=[Lv:Kv].
\]
We will denote by $d(L/K,v)$ the defect of $(L/K,v)$. This can be defined as follows. Let $L^h$ and $K^h$ be the \textit{henselizations} of $L$ and $K$, respectively, determined by a fixed extension of $v$ to $\overline{K}$. Then
\[
d(L/K,v):=\frac{[L^h:K^h]}{(vL:vK)\cdot [Lv:Kv]}.
\]

\section{The generation of an extension of valuation rings}
We start this section by describing the known results about the generation of extensions of valuation rings. For a finite valued field extension $(L/K,v)$, we are interested in whether $\VR_L$ is finitely, or essentially finitely, generated over $\VR_K$. This problem was first studied in \cite{Nov1}. There, the following was shown.

\begin{Teo}\cite[Theorem 1.3]{Nov1}
Assume that $L$ lies in the \emph{inertial field of $K$}. Then $\VR_L=\VR_K[\eta]_v$ for any given \emph{henselian generator} $\eta$ of $L/K$.
\end{Teo}

The above theorem implies that, in this particular case, $\VR_L$ is essentially finitely generated over $\VR_K$. Also in \cite{Nov1} it was shown that even if $L$ lies in the inertial field of $K$, it is not necessarily true that $\VR_L$ is finitely generated as an $\VR_K$-algebra. In the same work, conditions were presented conditions for this to be satisfied.

A natural reason to find natural generators of the extension of valuation rings is to compute the module of K\"ahler differentials of $\VR_L/\VR_K$. For this purpose, in \cite{CKR} and \cite{CK} the authors explicitly presented sets of generators for the cases they treated. More precisely, they assume that $L/K$ is Galois and of prime degree and used these properties to present sets of generators of $\VR_L$ over $\VR_K$.

The sets of generators  obtained in \cite{CKR} and \cite{CK} (in most of the cases) can be obtained from complete sets for $(L/K,v)$. Moreover, for any extension of valued fields with $e(L/K,v)=1$, any complete set for $(L/K,v)$ gives rise to a set generators of $\VR_L$ over $\VR_K$. More precisely, the following is true.

\begin{Prop}\cite[Proposition 3.5]{NS2023}\label{Corfinitelgenalg}
Let $(L/K,v)$ be a simple algebraic extension of valued fields and assume that $e(L/K,v)=1$. Take a generator $\eta$ of $L/K$ and consider the valuation $\nu_\eta$ of $K[x]$ defined by $v$ and $\eta$. For any complete set $\{Q_i\}_{i\in I}$ for $\nu_\eta$ and every $i\in I$ (with $\nu_\eta(Q_i)\neq \infty$) choose $a_i\in K$ such that $v(a_i)=\nu_\eta(Q_i)$. Then
\begin{equation}\label{formulagebfvalr}
\VR_L=\VR_K\left[\left.\frac{Q_i(\eta)}{a_i}\ \right|\ i\in I\right].
\end{equation}
\end{Prop}

When $e(L/K,v)\neq 1$, the situation is more complicated. The reason for that is that complete sets do not give rise (as before) to sets of generators. In \cite{CKR} and \cite{CK} the case of extensions of prime degree in which $e(L/K,v)\neq 1$ was also studied. This process was generalized in \cite{NS2023}.

\begin{Lema}\cite[Lemma 6.10]{NS2023}
If $(L/K,v)$ is pure in $\eta$ and $e(L/K,v)=[L:K]$, then
\begin{equation}
\VR_L=\VR_K\left[\left.\frac{\eta}a\ \right|\ a\in K\mbox{ and }va<v\eta\right].\label{eq:VRcase1}
\end{equation}
\end{Lema}

\subsection{About Theorem \ref{mainthem}}
We present now some general results that will be crucial in the proof of Theorem \ref{mainthem}. Let $\mu$ be any valuation of $K[x]$ and $v=\mu|_K$.

\begin{Lema}\label{generatimpor}
Assume that $\textbf Q=\{Q_i\}_{i\in I}$ is a complete set for $\mu$. Then for every $f\in K[x]$ there exist $a_1,\ldots, a_r\in K$ and $\lambda_1,\ldots,\lambda_r\in \N_0^I$ such that
\begin{equation}\label{Equanposrnalo}
f=\sum_{\ell=1}^r a_\ell \textbf{Q}^{\lambda_\ell}\mbox{ and }\mu(f)=\min_{1\leq \ell\leq r}\left\{\mu\left(a_\ell \textbf{Q}^{\lambda_\ell}\right)\right\}.
\end{equation}
\end{Lema}
\begin{proof}
We will use induction on the degree of $f$. If $\deg(f)=1$, then there exists $i\in I$ such that $\deg(Q_i)=1$ and $\mu_{Q_i}(f)=\mu(f)$. This means that there exist $a_1,a_2\in K$ such that
\[
f=a_1Q_i+a_2\mbox{ and }\mu(f)=\min\{\mu(a_1Q_i),\mu(a_2)\}
\]
and hence \eqref{Equanposrnalo} is satisfied for $f$.

Now consider an integer $n>1$ and assume that for every $f\in K[x]$, if $\deg(f)<n$, then there exist $\lambda_1,\ldots \lambda_r\in\N_0^I$ and $a_1,\ldots,a_r\in K$ such that \eqref{Equanposrnalo} is satisfied.

Take $f\in K[x]$ with $\deg(f)=n$. By our assumption on $\textbf{Q}$, there exists $i\in I$ such that
\[
\deg(Q_i)\leq\deg(f)\mbox{ and }\mu(f)=\mu_{Q_i}(f).
\]
This means that
\begin{equation}\label{equatruncation2}
f=f_0+f_1Q_i+\ldots+f_sQ_i^s,
\end{equation}
for some $f_0,\ldots,f_s\in K[x]$ with $\deg(f_\ell)<\deg(Q_i)$ for every $\ell$, $0\leq \ell\leq s$, and
\begin{equation}\label{equatruncation}
\mu(f)=\min_{0\leq \ell \leq s}\left\{\mu\left(f_\ell Q_i^\ell\right)\right\}.
\end{equation}

For each $\ell$, $0\leq \ell\leq s$, since $\deg(f_\ell)<\deg(Q_i)\leq\deg(f)=n$, there exist
\[
\lambda_{\ell,1},\ldots,\lambda_{\ell,r_\ell}\in\N_0^I\mbox{ and }a_{\ell,1},\ldots, a_{\ell,r_\ell}\in K
\]
such that
\begin{equation}\label{equatruncation4}
f_\ell=\sum_{k=1}^{r_\ell}a_{\ell,k}\textbf{Q}^{\lambda_{\ell,k}}\mbox{ and }\mu(f_\ell)=\min_{1\leq k\leq r_\ell}\left\{\mu\left(a_{\ell,k}\textbf{Q}^{\lambda_{\ell,k}}\right)\right\}.
\end{equation}
By \eqref{equatruncation2}, \eqref{equatruncation} and \eqref{equatruncation4} we deduce that 
\[
f=\sum_{\ell=0}^s\sum_{k=1}^{r_\ell}a_{\ell,k}\textbf{Q}^{\lambda'_{\ell,k}}\mbox{ and }\mu(f)=\min\limits_{\substack{0\leq \ell \leq s\\1\leq k\leq r_\ell}}\left\{\mu\left(a_{\ell,k}\textbf{Q}^{\lambda'_{\ell,k}}\right)\right\},
\]
where
\begin{displaymath}
\lambda'_{\ell,k}\left(j\right)=\left\{
\begin{array}{ll}
\lambda_{\ell,k}(j)+k&\mbox{ if }k=j\\[8pt]
\lambda_{\ell,k}(j)&\mbox{ if }k\neq j
\end{array}.
\right.
\end{displaymath}
This concludes our proof.

\end{proof}

Let $\Gamma$ be an ordered abelian group such that $\mu(K[x])\subseteq \Gamma$ and assume that there exist $\gamma_1,\ldots,\gamma_n\in \Gamma$ for which
\begin{equation}\label{eqieepsilondelta}
\Gamma=\bigcup_{\ell=1}^n\left(\gamma_\ell+v K\right).
\end{equation}
This implies that for every $f\in K[x]$ there exists $a\in K$ such that
\[
\mu\left(\frac{f}{a}\right)\in \{\gamma_1,\ldots,\gamma_n\}.
\]

\begin{Lema}\label{obssobregened}
Suppose that there exist $\gamma_1,\ldots,\gamma_n\in \Gamma$ satisfying \eqref{eqieepsilondelta} such that
\[
0=\gamma_1<\gamma_2<\ldots<\gamma_n<vK_{>0}.
\]
Suppose that $q\in K[x]$ is such that $\mu(q)\in \{\gamma_1,\ldots,\gamma_n\}$. For $a\in K$ if $\mu(aq)\geq 0$, then $a\in \VR_K$. 
\end{Lema}
\begin{proof}
Since $\mu(aq)\geq 0$, we have $v\left(a^{-1}\right)\leq\mu(q)$. Hence, $v(a)<0$ would imply that
\[
0< v\left(a^{-1}\right)\leq \mu(q)=\gamma_\ell,\mbox{ for some }\ell, 1\leq \ell\leq n,
\]
and this would be a contradiction to our assumption on the $\gamma_\ell$'s.

\end{proof}

\begin{Teo}\label{Theoremkeypol2}
Suppose that there exist $\gamma_1,\ldots,\gamma_n\in \Gamma$ satisfying \eqref{eqieepsilondelta} such that
\[
0=\gamma_1<\ldots<\gamma_n<vK_{>0}.
\]
Assume that $\{Q_i\}_{i\in I}$ is a complete set for $\mu$. For each $\lambda\in \N_0^{I}$ choose $a_\lambda$ such that
\[
\mu\left(\frac{\textbf Q^\lambda}{a_\lambda}\right)\in\{\gamma_1,\ldots,\gamma_n\}.
\]
Then the $\VR_K$-module
\[
B=\{f\in K[x]\mid \mu f\geq 0\}
\]
is generated by
\[
\left\{\left.\frac{\textbf{Q}^\lambda}{a_\lambda}\ \right|\ \lambda\in \N_0^I\right\}.
\]
\end{Teo}
\begin{proof}
For any $f\in B$ by Lemma \ref{generatimpor} there exist $\lambda_1,\ldots,\lambda_r\in \N_0^I$ and $a_1,\ldots, a_r\in K$ such that
\[
f=\sum_{\ell=1}^ra_\ell \textbf{Q}^{\lambda_\ell}\mbox{ and }0\leq \min_{1\leq \ell\leq r}\left\{\mu\left(a_\ell\textbf{Q}^{\lambda_\ell}\right)\right\}.
\]
This implies that 
\[
f=\sum_{\ell=1}^ra_\ell a_{\lambda_\ell} \frac{\textbf{Q}^{\lambda_\ell}}{a_{\lambda_\ell}}\mbox{ and }0\leq \min_{1\leq \ell\leq r}\left\{\mu\left( a_\ell a_{\lambda_\ell}\frac{\textbf{Q}^{\lambda_\ell}}{a_{\lambda_\ell}}\right)\right\}.
\]
Hence, for every $\ell$, $1\leq \ell\leq r$, we have
\[
0\leq\mu\left(a_\ell a_{\lambda_\ell}\frac{\textbf{Q}^{\lambda_\ell}}{a_{\lambda_\ell}}\right).
\]
Since $\mu\left(\frac{\textbf{Q}^{\lambda_\ell}}{a_{\lambda_\ell}}\right)\in \{\gamma_1,\ldots,\gamma_n\}$ by Lemma \ref{obssobregened} we deduce that $a_\ell a_{\lambda_\ell}\in \VR_K$. This concludes the proof.
\end{proof}

\begin{Prop}[Proposition 3.4 of \cite{Nov12}]\label{Propcharepsioln}
Let $\Gamma$ be an ordered abelian group and take $\Delta$ a subgroup of $\Gamma$ of finite index. If $e:=(\Gamma:\Delta)=\epsilon(\Gamma\mid\Delta)$, then there exist $\gamma_1,\ldots, \gamma_e\in \Gamma$ such that
\[
\Gamma=\bigcup_{i=1}^{e}\left(\gamma_i+\Delta\right)\mbox{ and }0=\gamma_1<\ldots<\gamma_e<\Delta_{>0}.
\] 
\end{Prop}

\begin{proof}[Proof of Theorem \ref{mainthem}]
Since $e:=e(L/K,v)=\epsilon(L/K,v)$, by Proposition \ref{Propcharepsioln} there exist $\gamma_1,\ldots, \gamma_e\in vL$ such that
\[
vL=\bigcup_{\ell=1}^{e}(\gamma_\ell+vK)\mbox{ and }0=\gamma_1<\ldots<\gamma_e<vK_{>0}.
\]

Let $\eta$ be a generator of the extension $L/K$. For every $b\in L$ there exists a polynomial $f(x)\in K[x]$ (with $\deg(f)<[L:K]$) such that $b=f(\eta)$. If $b\in \VR_L$, then $0\leq\nu_\eta f$ and hence $f\in B$ (as in Theorem \ref{Theoremkeypol2} for $\mu=\nu_\eta$). By Theorem \ref{Theoremkeypol2}, there exist $a_1,\ldots, a_r\in \VR_K$ and $\lambda_1,\ldots, \lambda_r\in \N_0^I$ such that
\[
b=f(\eta)=\sum_{\ell=1}^ra_j\frac{\textbf{Q}(\eta)^{\lambda_\ell}}{a_{\lambda_\ell}}.
\]
This concludes the proof.
\end{proof}

\section{Key polynomials and the defect formula}\label{keypolydefectformula}
Let $\mu$ be a valuation of $K[x]$ and fix an extension $\overline \mu$ of $\mu$ to $\overline K[x]$, where $\overline K$ is a fixed algebraic closure of $K$. For each $f\in K[x]$ we define
\[
\epsilon(f):=\max\{\overline \mu(x-a)\mid a\mbox{ is a root of }f\}.
\]
By \cite[Remark 3.2]{NovaMP} the value $\epsilon(f)$ does not depend on the extension $\overline \mu$ of $\mu$. A monic polynomial $Q\in K[x]$ is called \textbf{a key polynomial for $\mu$} if
\[
\deg(f)<\deg(Q)\Longrightarrow \epsilon(f)<\epsilon(Q)\mbox{  for all }f\in K[x].
\]
If $Q$ is a key polynomial for $\mu$, then $\mu_Q$ is a valuation (\cite[Proposition 2.6]{Novspivkeypol}). A \textbf{complete sequence of key polynomials} for $\mu$ is a set $\textbf{Q}=\{Q_i\}_{i\in I}$ such that $I$ is well-ordered, the map $i\mapsto Q_i$ is an order preserving map (i.e., for $i,j\in I$ we have $i<j\Lra \epsilon(Q_i)<\epsilon(Q_j)$), every $Q_i$ is a key polynomial and $\textbf{Q}$ is a complete set for $\mu$.

Let $\Gamma_\mu$ be the value group of $\mu$. The \textbf{graded ring of $\mu$} is defined as
\[
\mathcal G_\mu:=\bigoplus_{\gamma\in \Gamma_\mu}\{h\in K[x]\mid \mu(h)\geq \gamma\}/\{h\in K[x]\mid \mu(h)> \gamma\}.
\]
For $f\in K[x]$ for which $\mu(f)\neq \infty$, we define the \textbf{initial form} of $f$ in $\mathcal G_\mu$ by
\[
\inv_\mu(f):=f+\{h\in K[x]\mid \mu(h)> \mu(f)\}\in \mathcal G_\mu.
\]

For a key polynomial $Q$ for $\mu$ we can consider the graded ring of $\mu_Q$ which we denote by $\mathcal G_Q$ (instead of $\mathcal G_{\mu_Q}$). For $f\in K[x]$, with $\mu_Q(f)\neq \infty$, we denote $\inv_Q(f):=\inv_{\mu_Q}(f)$. Let
\[
R_Q:=\langle\{\inv_Q(f)\mid \deg(f)<\deg(Q)\}\rangle\mbox{ and }y_Q:=\inv_Q(Q)\in \mathcal G_Q.
\]
This means that $R_Q$ is the abelian subgroup of $\mathcal{G}_Q$ generated by the initial forms of polynomials of degree smaller than $\deg(Q)$. 
\begin{Prop}\cite[Proposition 4.5]{Nov1}
The set $R_Q$ is a subring of $\mathcal G_Q$, $y_Q$ is transcendental over $R_Q$ and
\[
\mathcal G_Q=R_Q[y_Q].
\]
\end{Prop}
In view of the previous proposition, for every $f\in K[x]$, with $\mu_Q(f)\neq \infty$, we can define the \textbf{degree of $f$ with respect to $Q$} as the degree of $\inv_Q(f)$ with respect to $y_Q$, i.e.,
\[
\deg_Q(f):=\deg_{y_Q}(\inv_Q(f)).
\]
\begin{Obs}\label{remakcalculed}
An easy way to see the number $\deg_Q(f)$ is the following. Let $f=f_0+f_1Q+\ldots+f_rQ^r$ be the $Q$-expansion of $f$. Then
\[
\deg_Q(f)=\max\left\{\ell\ \left|\right. \ \mu(f_\ell Q^\ell)=\mu_Q(f)\right\}. 
\]
In particular, if $\ell>d:=\deg_Q(f)$, then
\[
\mu(f_\ell Q^\ell)>\mu_Q(f)=\mu(f_d Q^d).
\]
\end{Obs}
For $m\in \N$, we say that $\Psi_m$ is a \textbf{plateau of key polynomials} for $\mu$ if $\Psi_m\neq \emptyset$ and $\mu(\Psi_m)$ does not have a maximum. If there exists $F\in K[x]$ such that
\begin{equation}\label{equmlimiekeypoly}
\mu_Q(F)<\mu(F)\mbox{ for every }Q\in \Psi_m,
\end{equation}
then any monic polynomial $F$ of smallest degree among polynomials satisfying \eqref{equmlimiekeypoly} is called a \textbf{limit key polynomial for }$\Psi_m$. In this case, it follows from \cite{Vaq0} that there exists $Q\in \Psi_m$ such that 
\begin{equation}\label{establekep}
\deg_{Q'}(F)=\deg_Q(F)\mbox{ for every }Q'\in\Psi_m\mbox{ with }\mu(Q')\geq \mu(Q).
\end{equation}
We define the \textbf{defect of $\Psi_m$} as
\[
d(\Psi_m):=\deg_{Q}(F)\mbox{ for any }Q\mbox{ satisfying \eqref{establekep}}.
\]

The next result is called \textbf{the defect formula}.
\begin{Teo}\cite[Theorem 6.14]{NN}\label{defectformu}
Let $(L/K,v)$ be a simple algebraic valued field extension and fix a generator $\eta$ of $L/K$. Consider the valuation $\nu_\eta$ of $K[x]$ defined by $v$ and $\eta$ and let $m_1,\ldots,m_r\in\N$ be all the natural numbers $m$ for which $\Psi_m$ is a plateau for $\nu_\eta$. Then
\[
d(L/K,v)=\prod_{\ell=1}^r d(\Psi_{m_\ell}).
\]
\end{Teo}

For a polynomial $f\in K[x]$ and $j\in\N_0$ we denote by $\partial_jf$ the $j$-\emph{Hasse derivative} of $f$. Assume that $(L/K,v)$ is pure in $\eta$ and let $g$ be the minimal polynomial of $\eta$ over $K$. By \cite[Corollary 3.4]{Novspivkeypol} there exists $c\in K$ such that for every $c'\in K$, if $v(\eta-c')\geq v(\eta-c)$, then
\begin{equation}\label{eqsnkaplansmmine}
\beta_\ell:=v\left(\partial_\ell g(c)\right)=v\left(\partial_\ell g(c')\right)\mbox{ for every }\ell, 1\leq\ell\leq \deg(g).
\end{equation}

\begin{Prop}\label{propdefectcasepure}
Assume that $(L/K,v)$ is pure in $\eta$ and let $g$ be the minimal polynomial of $\eta$ over $K$. Set $d:=d(L/K,v)$. Assume that $\Psi_1$ is a plateau for $\nu_\eta$. Then there exists $c\in K$ such that for every $c'\in K$, if $v(\eta-c')\geq v(\eta-c)$, then
\begin{equation}\label{eqauatidiffe1}
\beta_d+d v(\eta-c')<\beta_\ell+\ell v(\eta-c')\mbox{ for every }\ell, d<\ell\leq \deg(g).
\end{equation}
\end{Prop}
\begin{proof}
Since $(L/K,v)$ is pure in $\eta$, for every $m$, $1<m<n:=\deg(g)=[L:K]$, we have $\Psi_m=\emptyset$. In particular, $g$ is a limit key polynomial for $\Psi_1$. For every $c\in K$, the $(x-c)$-expansion of $g$ is
\[
g=g(c)+\partial g(c)(x-c)+\ldots+\partial_ng(c)(x-c)^n.
\]
By the defect formula $d=d(\Psi_1)$. This implies that there exists $c\in K$ such that if $v(\eta-c')\geq v(\eta-c)$, then $d=\deg_{x-c'}g$. We can take $c\in K$ so that \eqref{eqsnkaplansmmine} is satisfied for every $c'\in K$ for which $v(\eta-c')\geq v(\eta-c)$. By Remark \ref{remakcalculed}, this implies that for $\ell$, $d<\ell\leq n$, we have
\[
\nu_\eta\left(\partial_dg(c')(x-c')^d\right)<\nu_\eta\left(\partial_\ell g(c')(x-c')^\ell\right).
\]
This concludes the proof of the proposition.
\end{proof}

\section{About Knaf's conjecture}
In this section we use Theorem \ref{mainthem} and the theory of key polynomials to show Knaf's conjecture for pure extensions. 
\begin{Lema}\label{corsequenclpo}
If $(L/K,v)$ is pure in $\eta$, then $\{x-K\}\cup\{g\}$ is a complete set for $\nu_\eta$ where $g$ is the minimal polynomial of $\eta$ over $K$. In particular, for every set $\{c_i\}_{i\in I}\subseteq K$ such that $\{v(\eta-c_i)\}_{i\in I}$ is well-ordered and cofinal in $v(\eta-K)$, the set $\{x-c_i\}_{i\in I}\cup\{g\}$ is a complete sequence of key polynomials for $\nu_\eta$.
\end{Lema}
\begin{proof}
Take any $f\in K[x]$. If $f=f_0+f_1g+\ldots+f_rg^r$ is the $g$-expansion of $f$, then
\[
\nu_\eta(f)=(\nu_\eta)_g(f)=\nu_\eta(f_0).
\]
Hence, our result is trivially satisfied if $\deg(f)\geq \deg(g)$. Assume that $\deg(f)<\deg(g)$. By hypothesis, $\Psi_m=\emptyset$ for every $m$, $1<m\leq \deg(f)$. Hence, by \cite[Lemma 2.11]{Novspivkeypol}, there exists $c\in K$ such that $(\nu_\eta)_{x-c}(f)=\nu_\eta(f)$. The second statement follows from the fact that if $v(\eta-c_i)>v(\eta-c)$, then by \cite[Proposition 2.10 \textbf{(iii)}]{Novspivkeypol}
\[
(\nu_\eta)_{x-c}(f)\leq (\nu_\eta)_{x-c_i}(f)\leq \nu_\eta(f).
\]
\end{proof} 
\begin{Cor}
As in the notation of Lemma \ref{corsequenclpo}, for every $f\in K[x]$, $\deg(f)<[L:K]$, we have
\begin{equation}\label{equabtomfualteL}
vf(\eta)=\min_{0\leq \ell\leq \deg(f)}\{v\left(\partial_\ell f(c_i)(\eta-c_i)^\ell\right)\}\mbox{ for some }i\in I. 
\end{equation}
\end{Cor}
\begin{proof}
For $i\in I$ such that $(\nu_\eta)_{x-c_i}(f)=(\nu_\eta)(f)$, since the $(x-c_i)$-expansion of $f$ is
\[
f(x)=\sum_{\ell=0}^{\deg(f)}\partial_\ell f(c_i)(x-c_i)^\ell
\]
this implies that
\[
vf(\eta)=\nu_\eta f=\min_{0\leq \ell\leq \deg(f)}\left\{\nu_\eta\left(\partial_\ell f(c_i)(x-c_i)^\ell\right)\right\}=\min_{0\leq \ell\leq \deg(f)}\left\{v\left(\partial_\ell f(c_i)(\eta-c_i)^\ell \right)\right\}
\]
\end{proof}
\begin{Lema}\label{generationofideals}
If $(L/K,v)$ is pure in $\eta$ and $e(L/K,v)=1$, then
\begin{equation}\label{equatinequali}
\VR_L=\VR_K\left[\left. \frac{\eta-c}{a}\ \right|\ c,a\in K\mbox{ and }va\leq v(\eta-c)\right]. 
\end{equation}
\end{Lema}
\begin{proof}
The right hand side of \eqref{equatinequali} is clearly contained in its left hand side. For any $b\in L$, write $b=f(\eta)$ for some $f\in K[x]$, $\deg(f)<n$. Since $(L/K,v)$ is pure in $\eta$, by \eqref{equabtomfualteL} there exists $c\in K$ such that
\begin{equation}\label{eqantinoesxpalnnyb}
vb=\min_{0\leq \ell\leq \deg(f)}\left\{v\left(\partial_\ell f(c)(\eta-c)^\ell\right)\right\}. 
\end{equation}
Take $a\in K$ such that $v(\eta-c)=va$. Then
\begin{equation}\label{eqnatuonsepxpi}
b=f(\eta)=\sum_{\ell=0}^{\deg(f)}\partial_\ell f(c)(\eta-c)^\ell=\sum_{\ell=0}^{\deg(f)}\partial_\ell f(c)a^\ell\left(\frac{\eta-c}{a}\right)^\ell.
\end{equation}
If $b\in \VR_L$, then by \eqref{eqantinoesxpalnnyb} we deduce that
\[
0\leq vb=\min_{0\leq \ell\leq \deg(f)}\left\{v\left(\partial_\ell f(c)a^\ell\right)\right\}
\]
and by \eqref{eqnatuonsepxpi} we conclude that
\[
b\in \VR_K\left[\frac{\eta-c}{a}\right].
\]
\end{proof}
The next result is a particular case of Proposition \ref{Propimportan}. We present its proof here in order to illustrate our method.

\begin{Prop}\label{pureeigualepsil}
Assume that $(L/K,v)$ is pure in $\eta$ and $e(L/K,v)=1$. If $d(L/K,v)=1$, then
\[
\VR_L=\VR_K\left[\frac{\eta-\overline c}{a}\right]_v
\]
for some $a,\overline c\in K$. 
\end{Prop}
\begin{proof}
If $\nu_\eta(\Psi_1)$ has a maximum $\nu_\eta(x-\overline c)$, then by Lemma \ref{generationofideals} we have
\[
\VR_L=\VR_K\left[\frac{\eta-\overline c}{a}\right]
\]
for any $a\in K$ with $v(\eta-\overline{c})=va$. Hence the result follows.

Assume now that $\Psi_1$ is a plateau. By the defect formula there exists $\overline c\in K$ such that if $v(\eta-c)\geq v(\eta-\overline c)$, then $\deg_{x-c}g=1$. Replacing $\eta$ by $\eta-\overline c$, we can assume that $\overline c=0$.

Assume that for every $c\in K$ with $v(\eta-c)\geq v\eta$ we have
\[
\beta_\ell:=v\left(\partial_\ell g(0)\right)=v\left(\partial_\ell g(c)\right)\mbox{ for every }\ell, 1\leq \ell\leq n.
\]
Since
\begin{equation}\label{formulaforg}
g=g(c)+\partial g(c)(x-c)+\ldots+\partial_ng(c)(x-c)^n
\end{equation}
and $d(L/K,v)=1$ we deduce by Proposition \ref{propdefectcasepure} that
\begin{equation}\label{eqauatidiffe}
v(g(c))=\beta_1+v(\eta-c)<\beta_\ell+\ell v(\eta-c)\mbox{ for every }\ell, 1<\ell\leq n.
\end{equation}
Take $a\in K$ such that $v\eta=va$. We will show that for every $c\in K$ with $v(\eta-c)$ large enough, and $a'\in K$ with $va'\leq v(\eta-c)$ we have
\[
\frac{\eta-c}{a'}\in \VR_K\left[\frac{\eta}{a}\right ]_v.
\]
This together with Lemma \ref{generationofideals} will imply the result.

For each $\ell$, $0\leq \ell\leq n$, set
\[
b_\ell=\frac{\partial_\ell g(c)}{\partial g(c)}.
\]
From \eqref{formulaforg}, we deduce that that
\[
\eta-c=-\frac{b_0}{1+b_2(\eta-c)+\ldots+ b_n(\eta-c)^{n-1}}.
\]
Set
\[
h:=1+b_2(\eta-c)+\ldots+b_n(\eta-c)^{n-1}.
\]
It remains to show that
\begin{equation}\label{hisqhatenat}
vh=0\mbox{ and }h\in \VR_K\left[\frac \eta a\right].
\end{equation}
Indeed, if this is true, then $v(\eta-c)=vb_0$ and consequently
\[
\frac{\eta-c}{a'}=-\frac{b_0/a'}{1+b_2(\eta-c)+\ldots+ b_n(\eta-c)^{n-1}}\in \VR_K\left[\frac{\eta}{a}\right]_v.
\]

By the binomial expansion we have $h=1+\overline h$ where $\overline h$ is a sum of terms of the form
\[
{\ell-1 \choose j} b_\ell\eta^jc^{\ell-j-1}={\ell-1 \choose j} b_\ell a^jc^{\ell-j-1}\left(\frac{\eta}{a}\right)^j
\]
for some $\ell, 2\leq \ell\leq n$ and $j,0\leq j\leq \ell-1$. Since $v\eta=vc=va$ and $v(b_\ell)=\beta_\ell-\beta_1$, by \eqref{eqauatidiffe}, we have
\[
v\left({\ell-1 \choose j} b_\ell a^jc^{\ell-j-1}\right)\geq vb_\ell+(\ell-1)v\eta=\beta_\ell-\beta_1+(\ell-1)v\eta>0.
\]
Hence, $v\overline h>0$ and we deduce \eqref{hisqhatenat}.
\end{proof}

We will now prove Proposition \ref{Propimportan} which is a generalization of the previous result to the case $\epsilon(L/K,v)=e(L/K,v)$.  

\begin{proof}[Proof of Proposition \ref{Propimportan}]
Consider the notation as in the proof of Proposition \ref{pureeigualepsil}. As in Proposition \ref{pureeigualepsil}, if $\nu_\eta(\Psi_1)$ has a maximum, then the result follows immediately from Theorem \ref{mainthem}. Hence, assume that $\Psi_1$ is a plateau.
 
For each $\ell$, $1\leq \ell<e$, take $a_\ell\in K$ such that
\[
\frac{\eta^\ell}{a_\ell}\in \{\gamma_1,\ldots, \gamma_{e-1}\}.
\]
Also, choose $a_e\in K$ such that $va_e=ev\eta$. For each $j$, $1\leq j\leq n$, write
\begin{equation}
\eta^j=a_ra_e^s\frac{\eta^r}{a_r}\left(\frac{\eta^e}{a_e}\right)^s\mbox{ where }j=r+se, 0\leq r<e.
\end{equation}

From \eqref{formulaforg}, we deduce that
\begin{equation}\label{equaimtenepsilowe}
\eta-c=-\frac{b_0}{1+b_2(\eta-c)+\ldots+ b_n(\eta-c)^{n-1}}.
\end{equation}
Set
\[
h:=1+b_2(\eta-c)+\ldots+b_n(\eta-c)^{n-1}.
\]
We claim that
\begin{equation}\label{hisqhatenat1}
vh=0\mbox{ and }h\in \VR_K\left[\left. \frac{\eta^\ell}{a_\ell}\ \right|\ 1\leq \ell\leq e\right ].
\end{equation}
Indeed, by the binomial expansion we have $h=1+\overline h$ where $\overline h$ is a sum of terms of the form
\[
{\ell-1 \choose j} b_\ell\eta^jc^{\ell-j-1}={\ell-1 \choose j} b_\ell c^{\ell-j-1}a_ra_e^s\frac{\eta^r}{a_r}\left(\frac{\eta^e}{a_e}\right)^s
\]
for some $\ell, 2\leq \ell\leq n$ and $j,0\leq j\leq \ell-1$. Since $v\eta=vc$ and $v(b_\ell)=\beta_\ell-\beta_1$, by \eqref{eqauatidiffe}, we have
\[
v\left({\ell-1 \choose j} b_\ell c^{\ell-j-1}a_ra_e^s\right)\geq vb_\ell+(\ell-1)v\eta+v(a_r)-rv\eta>-\gamma_r.
\]
In particular,
\[
{\ell-1 \choose j} b_\ell c^{\ell-j-1}a_ra_e^s\in \VR_K
\]
and consequently $v\overline h>0$. Therefore, we deduce \eqref{hisqhatenat1}.

Take a sequence of key polynomials $\textbf{Q}=\{x-c_i\}_{i\in I}\cup\{g\}$ for $\nu_\eta$. For every $\lambda\in  \N_0^I$, by \eqref{equaimtenepsilowe} and \eqref{hisqhatenat1} there exist
\[
b_\lambda\in K\mbox{ and }H_\lambda\in \VR_K\left[\left. \frac{\eta^\ell}{a_\ell}\ \right|\ 1\leq \ell\leq e\right ]
\]
such that
\[
vH_\lambda=0\mbox{ and }\textbf{Q}(\eta)^\lambda=\frac{b_\lambda}{H_\lambda}.
\]
In particular, for every $a_\lambda$ such that
\[
v\left(\frac{\textbf{Q}(\eta)^\lambda}{a_\lambda}\right)\in\{\gamma_1,\ldots,\gamma_e\}
\]
we deduce that
\[
\frac{\textbf{Q}(\eta)^\lambda}{a_\lambda}=\frac{b_\lambda/a_\lambda}{H_\lambda}\in \VR_K\left[\left.\frac{\eta^\ell}{a_\ell}\  \right|\ 1\leq\ell\leq e\right ]_v.
\]
This together with Theorem \ref{mainthem} implies the result.
\end{proof}


\begin{thebibliography}{99}

\bibitem{CKR} S. D. Cutkosky, F.-V. Kuhlmann and  A. Rzepka, \textit{Characterizations of Galois extensions with independent defect}, arXiv:2305.10022 (2023).

\bibitem{CK} S. D. Cutkosky and F.-V. Kuhlmann, \emph{Kähler differentials of extensions of valuation rings and deeply ramified fields},  arXiv:2306.04967v1 (2023).

\bibitem{CutKnaf} S. D. Cutkosky, \emph{Local Uniformization of Abhyankar Valuations}, Michigan Math. J. \textbf{71 (4)} (2022), 859--891.

\bibitem{Nov12} S. D. Cutkosky and J. Novacoski, \textit{Essentially finite generation of valuation rings in terms of classical invariants}, Math. Nachrichten \textbf{294} (2021), 15--37.

\bibitem{Datta2021} R. Datta, \emph{Essential finite generation of extensions of valuation rings}, Math. Nachrichten \textbf{296} (2021), 1041--1055.

\bibitem{Nov1} F.-V. Kuhlmann and J. Novacoski, \textit{Henselian elements}, J. Algebra \textbf{418} (2014), 44--65.

\bibitem{NN} E. Nart and J. Novacoski, \textit{The defect formula}, Adv. Math. \textbf{428} (2023), 109153.

\bibitem{Nov1} J. Novacoski, \textit{On MacLane-Vaqui\'e key polynomials}, Journal of Pure and Applied Algebra Volume 225, Issue 8 (2021), 106644.

\bibitem{NovaMP} J. Novacoski, \textit{Key polynomials and minimal pairs}, J. Algebra \textbf{523} (2019), 1--14.

\bibitem{NS2023} J. Novacoski and M. Spivakovsky, \textit{K\"ahler differentials, pure extensions and minimal key polynomials}, arXiv:2311.14322 (2023).

\bibitem{Novspivkeypol} J. Novacoski and M. Spivakovsky, \textit{Key polynomials and pseudo-convergent sequences}, J. Algebra \textbf{495} (2018), 199--219.

\bibitem{Nov4} J. Novacoski and M. Spivakovsky, \textit{On the local uniformization problem}, Algebra, Logic and Number Theory, Banach Center Publ. \textbf{108} (2016), 231--238.

\bibitem{Vaq0}M. Vaqui\'e, \emph{Famille admisse associ\'ee \`a une  valuation de $K[x]$}, Singularit\'es Franco-Japonaises, S\'eminaires et Congr\'es 10, SMF, Paris (2005), Actes du colloque franco-japonais, juillet 2002, \'edit\'e par Jean-Paul Brasselet et Tatsuo Suwa, 391--428.
\end{thebibliography}
\end{document}